\documentclass[11pt]{article}
\usepackage[letterpaper,margin=1in]{geometry}
\usepackage{amsmath,amssymb,amsthm}
\usepackage[hidelinks]{hyperref}
\numberwithin{equation}{section}
\newtheorem{theorem}{Theorem}[section]
\newtheorem{proposition}[theorem]{Proposition}
\newtheorem{corollary}[theorem]{Corollary}
\theoremstyle{remark}

\newcommand{\R}{\mathbb R}
\newcommand{\dd}{\,\mathrm d}
\newcommand{\Rea}{\operatorname{Re}}
\title{Finite-core Volterra reductions for the Riemann phase kernel\thanks{Revision of arXiv:2606.29555.}}
\author{Marvin B. Freedman\thanks{Independent Researcher. \texttt{marvinbfreedman@gmail.com}.}}
\date{December 2026}

\begin{document}
\maketitle
\begin{abstract}
We derive exact same-sign formulas for a Riemann theta phase kernel and give a
Volterra representation with its boundary and integrated tail terms. The
formulas apply to finite theta cores, and the boundary-plus-tail identity
also holds for the full theta series. We give an outward-rounded
fourteen-node counterexample to nonnegativity of the exact endpoint-B form
on the active closed-trace core in the original application. A rectangle enclosure gives the same strict sign for an explicit smooth
witness supported outside the trace interval. As a result, the original energy is
unbounded below on every nonempty affine trace fiber that retains this
vertical direction. We specify a coercive modification and its changed
Euler equation. A density argument gives a finite shifted-$\Xi$ evaluation test
whose Fourier restriction has strictly negative endpoint-B energy. Positivity of the original full-coordinate Weyl kernel
and the zeta-facing application remain unresolved.
\end{abstract}

\section{Results and correction to version 1}

This revision supersedes the completed-domain closure claims for the original
application in version~1 \cite{v1}. The exact formulas for the theta modes and
the coordinate kernel are retained. The endpoint-B form used in the
continuum argument has a negative direction on the stated trace kernel. The
argument in version~1, \S15.7 and Appendix~C, does not establish
source-inactive domination on that exact domain, the completed Green-lift
contraction, the consequent quotient-to-original lift, or uniform Weyl
positivity. The original zeta-facing bridge also remains open. The August
archive records the strict obstruction and the separate augmented formulation
\cite{archive,companion}.

I thank a reader for pointing out the gap in the completed-domain Green-lift
contraction argument in Appendix~C of version~1.

Section~\ref{sec:endpoint} defines the endpoint-B form, identifies the trace
core, and proves its failure of nonnegativity. Section~\ref{sec:volterra}
derives the coordinate formulas with a common normalization and all mode
pairs included. Section~\ref{sec:scope} gives the revised Euler equation.
The supplement (Section~\ref{sec:evaluations}) identifies the shifted-$\Xi$
tests and the sign required of their full energy. The endpoint-B form and the original full-coordinate kernel
are distinct objects throughout.

\section{The endpoint-B obstruction and fiber minimization}\label{sec:endpoint}

Use $s,t\in[0,8]$, $r\geq0$, $x=e^r$, $c=\pi$, and $\lambda_s=ce^s$.
Write
\begin{equation}\label{eq:end-features}
\begin{aligned}
 F(\lambda,r)&=(1-\lambda x)e^{-\lambda(x-1)},\\
 v_s(r)&=\frac{F(\lambda_s,r)-F(c,r)}{\lambda_s-c},&
 w_s(r)&=\frac{s F(\lambda_s,r)}{\lambda_s-c}.
\end{aligned}
\end{equation}
At $s=0$ use the values $v_0=\partial_\lambda F(c,\cdot)$ and
$w_0=F(c,\cdot)/c$. The kernel and its form are
\begin{align}
 K_B(s,t)&=\int_0^\infty e^{5r/2}
   \left(r v_s v_t+\tfrac12 v_s w_t+\tfrac12 w_s v_t\right)\dd r,
   \label{eq:endkernel}\\
 \mathfrak a_B[f,g]&=\int_0^8\!\int_0^8
          K_B(s,t) f(t)\overline{g(s)}\dd t\dd s,
 \qquad \mathfrak a_B[f]=\mathfrak a_B[f,f].\label{eq:endform}
\end{align}
These are the endpoint-B differential features in version~1, after the
change of variables $x=1+\tau=e^r$ and $x\partial_\tau=\partial_r$.

The active trace on $I=[0.02,0.545]$ has the form
\[
 (Rf)(u)=\sum_{k=0}^8 \frac{e_k(u)}{k!} f^{(k)}(u),\qquad u\in I,
 \qquad N_0=\{f\in C_c^\infty(0,8): Rf=0\ \text{a.e. on } I\}.
\]
The theorem works for any bounded measurable coefficients $e_k$, in
particular for the moving jet row in the endpoint model. It uses only
locality on $I$. No choice of a chart for the leading coefficient is needed.

\begin{theorem}[Exterior negative direction]\label{thm:negative}
The kernel \eqref{eq:endkernel} defines a bounded Hermitian form on $L^2(0,8)$.
There exists a real $h\in C_c^\infty((0.545,8))\subset N_0$ with
$-1.119\cdot10^{-28}<\mathfrak a_B[h]<-1.118\cdot10^{-28}$. In particular, nonnegativity fails on the smooth
closed-trace core.
\end{theorem}

\begin{corollary}[Failure of the original minimum-energy lift]\label{cor:fiber}
Let $V$ be a linear space contained in $L^2(0,8)$ containing the witness $h$ of
Theorem~\ref{thm:negative}. Let $R_V$ be a linear trace on $V$ with $R_Vh=0$.
Then on every nonempty affine fiber $\{f\in V:R_Vf=y\}$ the infimum of the
original energy is $-\infty$.
\end{corollary}

\begin{proof}[Proof of the corollary]
For a point $f$ in the fiber and $t\in\R$, the point $f+th$ lies in the same
fiber and its energy is
\[
 \mathfrak a_B[f]+2t\Rea\mathfrak a_B[f,h]+t^2\mathfrak a_B[h].
\]
The quadratic coefficient is strictly negative.
\end{proof}

\subsection{Continuity and signed feature form}

For $x=e^r$, differentiation of $F$ with respect to $\lambda$ and the
mean-value integral for the divided difference give, on $0\le s\le 8$,
\[
 |v_s(r)|\le C(1+x)^2e^{-c(x-1)},\qquad
 |w_s(r)|\le C(1+x)e^{-c(x-1)}.
\]
The second estimate uses $s/(e^s-1)\le1$. Thus all three terms in the
kernel are dominated in absolute value by an integrable multiple of
$(1+r)e^{5r/2}(1+e^r)^4e^{-2c(e^r-1)}$. The extensions in
\eqref{eq:end-features} are continuous, so $K_B$ is continuous on the
compact square and real symmetric. Its integral operator $T_B$ on $L^2(0,8)$
is Hilbert--Schmidt. In particular the form restricted to any smooth core
has the bounded extension on its $L^2$ closure.

Set $V_f=\int_0^8 v_s f(s)\dd s$ and $W_f=\int_0^8 w_s f(s)\dd s$.
The same dominance and Fubini's theorem yield
\begin{equation}\label{eq:signed}
 \mathfrak a_B[f]=\int_0^\infty e^{5r/2}
       \left(r|V_f|^2+\Rea(V_f\overline{W_f})\right)\dd r .
\end{equation}
The real symmetric matrix $\left(\begin{smallmatrix}r&1/2\\1/2&0\end{smallmatrix}\right)$
has eigenvalues $(r\pm\sqrt{r^2+1})/2$. Taking the weighted spectral components
defines maps $G_+$ and $G_-$ with
$\mathfrak a_B[f]=\|G_+f\|_2^2-\|G_-f\|_2^2$.
On any linear core, nonnegativity of this form is equivalent to the existence
of a contraction taking $G_+f$ to $G_-f$ on the closure of the first range.
Indeed nonnegativity gives $\|G_-f\|\leq\|G_+f\|$, making the map well-defined
and contractive on the range. Such a contraction implies the inequality in
the reverse direction. The witness in Theorem~\ref{thm:negative} excludes
this contraction on $N_0$ for the exact spectral components of \eqref{eq:signed}.

\subsection{Outward-rounded witness}

For $a>0$ write
\[
 J_p(a)=\int_1^\infty x^p e^{-a(x-1)}\dd x
        =e^a a^{-p-1}\Gamma(p+1,a),\qquad
 J_p^{[1]}(a)=\partial_p J_p(a).
\]
Define $M_0(\alpha,\beta)$ by evaluating
\begin{equation}\label{eq:moments}
 J_{3/2}(a)-aJ_{5/2}(a)+\alpha\beta J_{7/2}(a),\qquad a=\alpha+\beta,
\end{equation}
and define $M_1(\alpha,\beta)$ by replacing each $J_p$ by $J_p^{[1]}$.
After $x=e^r$, the weight $e^{5r/2}\dd r$ becomes $x^{3/2}\dd x$. Expanding
$F(\alpha,r)F(\beta,r)$ gives the polynomial
$1-(\alpha+\beta)x+\alpha\beta x^2$, yielding \eqref{eq:moments} and the
logarithmic moment. For $\alpha=\lambda_s$, $\beta=\lambda_t$ and $s,t>0$,
direct substitution into \eqref{eq:endkernel} gives
\begin{align}
 (\alpha-c)(\beta-c)K_B(s,t)={}&M_1(\alpha,\beta)-M_1(c,\beta)
   -M_1(\alpha,c)+M_1(c,c)\notag\\
 &+\tfrac t2\bigl(M_0(\alpha,\beta)-M_0(c,\beta)\bigr)\notag\\
 &+\tfrac s2\bigl(M_0(\alpha,\beta)-M_0(\alpha,c)\bigr).
 \label{eq:checked-matrix}
\end{align}

For $\varepsilon=2^{-120}$, the centered difference in $p$ has error at most
$\varepsilon^2 J_{p+\varepsilon+1}(a)$. In fact the third derivative is the
integral with an extra factor $(\log x)^3$. For $x\ge1$,
$(\log x)^3\le6x$, and $x^q\le x^{p+\varepsilon}$ on
$p-\varepsilon\le q\le p+\varepsilon$. Taylor's remainder bounds the difference
error by $\varepsilon^2/6$ times the supremum of this derivative. The upper
incomplete gamma evaluates the full integral over $[1,\infty)$.

The ancillary checker \texttt{verify\_endpoint\_b.py} contains the fixed
decimal coefficients $d_j$ and evaluates \eqref{eq:checked-matrix} by Arb
ball arithmetic \cite{arb}. At 768 bits it certifies
\begin{equation}\label{eq:negative-ball}
 -1.119\cdot10^{-28}<
 \sum_{i,j=0}^{13}d_i d_j K_B(s_i,s_j)
 <-1.118\cdot10^{-28},\qquad
 s_j=\frac35+\frac{3j}{50}.
\end{equation}
The coefficients are rounded versions of the archived witness and are
interpreted as exact decimals in this calculation. The sign check uses
only the kernel and the fixed vector.

\begin{proof}[Completion of the proof of Theorem~\ref{thm:negative}]
Fix $\eta=2^{-320}$ and $B_i=[s_i-\eta,s_i+\eta]$. Define
\[
 \begin{aligned}
 E(u)&=\begin{cases}e^{-1/u},&u>0,\\0,&u\le0,\end{cases}
 & b_i(x)&=E\!\left(1-\left(\frac{x-s_i}{\eta}\right)^2\right),\\
 p_i(x)&=\frac{b_i(x)}{\int_{\R} b_i(v)\,\dd v},
 & h(x)&=\sum_{i=0}^{13} d_i p_i(x).
\end{aligned}
\]
Each $b_i$ is smooth, vanishes off $B_i$ and is positive at $s_i$.
Thus $p_i\ge0$ has unit integral. The intervals lie in $[0.55,2]\subset
(0.545,8)$, so $h$ is a real smooth compactly supported test and every
derivative of $h$ vanishes on the active trace interval. In particular $Rh=0$.

For each ordered pair $(i,j)$ the checker evaluates \eqref{eq:checked-matrix}
with $s\in B_i$ and $t\in B_j$ by outward-rounded interval arithmetic.
Writing $L_{ij}=L_{ji}$ for the integers in the supplementary table, it verifies
\[
 L_{ij}\le 10^{40}K_B(s,t)\le L_{ij}+1
 \qquad ((s,t)\in B_i\times B_j).
\]
These 196 checks use the whole rectangles, including their interiors.
The entries
\[
 A_{ij}=\int_0^8\!\int_0^8 K_B(s,t)p_i(s)p_j(t)\,\dd t\,\dd s
\]
obey the same bounds because the weight is nonnegative, is supported on
the rectangle, and has integral one. Expanding the finite sum in
\eqref{eq:endform} gives $\mathfrak a_B[h]=\sum_{i,j}d_i d_j A_{ij}$.
For either sign of $d_i d_j$, the two products with the endpoints of
the $(i,j)$ cell bound that summand. Summing the smaller products and
the larger products separately, using the exact decimal coefficients,
gives
\[
 -1.119\cdot 10^{-28}<\mathfrak a_B[h]<-1.118\cdot 10^{-28}.
\]
The supplement reproduces the rectangle enclosures in Arb and checks the
averaging argument and signed sum in Lean. Corollary~\ref{cor:fiber} applies
on each admissible space containing this $h$ with $R_Vh=0$.
\end{proof}

\section{Exact same-sign Volterra identities}\label{sec:volterra}

For $t\ge0$ let $c_n=\pi n^2$ and set
\begin{equation}\label{eq:theta}
 b_n(t)=e^{t/2-c_ne^{2t}},\qquad
 \phi_n(t)=(\partial_t^2-\tfrac14)b_n(t)
          =(4c_n^2e^{4t}-6c_ne^{2t})b_n(t).
\end{equation}
Use the positive-side sum $\Phi=\sum_{n\ge1}\phi_n$ and its even extension.
The same formulas below also use a finite core in place of $\Phi$. For an
even profile $\phi$ so obtained, use the version~1 coordinate normalization
\begin{equation}\label{eq:coordinate}
 K_{\phi,\omega}(x,y)=\frac12\int_{|(x+y)/2|}^\infty
 z\cosh(2\omega z)\phi\!\left(z+\frac{x-y}{2}\right)
                   \phi\!\left(z-\frac{x-y}{2}\right)\dd z .
\end{equation}
The identities use real $\omega$, with $|\omega|<1/2$ the original target.

\begin{proposition}[Finite same-sign formula]\label{prop:gamma}
Let $\phi(t)=\sum_i A_i e^{\lambda_i t-c_i e^{2t}}$ on $t\ge0$, with $c_i>0$
and real $A_i,\lambda_i$. For $x,y\ge0$, define $\alpha_{ij}=c_i e^{2x}+c_j e^{2y}$
and $p_{ij,\epsilon}=(\lambda_i+\lambda_j+2\epsilon\omega)/2$. Then
\begin{align}
 K_{\phi,\omega}(x,y)
 &=\frac1{16}\sum_{\epsilon=\pm1}\sum_{i,j}
 A_i A_j e^{(\lambda_i+\epsilon\omega)x+(\lambda_j+\epsilon\omega)y}\notag\\
 &\quad\times\bigl(\partial_p I_0(\alpha_{ij},p_{ij,\epsilon})
             +(x+y)I_0(\alpha_{ij},p_{ij,\epsilon})\bigr),\label{eq:finite-gamma}
\end{align}
where $I_0(a,p)=\int_1^\infty q^{p-1}e^{-a q}\dd q=a^{-p}\Gamma(p,a)$.
The derivative is with respect to the second argument.
\end{proposition}

\begin{proof}
In \eqref{eq:coordinate} take $z=(u+x+y)/2$. The two profile arguments
become $x+u/2$ and $y+u/2$, and $\frac12 z\dd z=(u+x+y)\dd u/8$.
Expand both cores and $\cosh(\omega(u+x+y))$ into the two exponentials.
The substitution $q=e^u$ yields \eqref{eq:finite-gamma}, with $u$ giving
the logarithmic moment. Differentiation under the $q$ integral follows from
exponential decay for $a>0$.
\end{proof}

\begin{proposition}[Full-theta boundary and tail]\label{prop:volterra}
Let $(\nu_n)_{n\ge1}$ be a bounded nonnegative sequence with at least one
positive term, and let $\phi=\sum_n\nu_n\phi_n$ on the positive side. Set,
for $v\ge0$,
\begin{align*}
 B(v)&=\sum_n\nu_n e^{v/4-c_n e^v},
 &\Psi(v)&=4B''(v)-\tfrac14 B(v)=\phi(v/2),\\
 a(v)&=-B'(v)/B(v), &\mu(v)&=\Psi(v)/B(v).
\end{align*}
For $s,t\ge0$ let $U=(s+t)/2$ and put
\[
 C_\omega(s,t)=\frac{U\cosh(\omega U)\mu(s)\mu(t)}{a(s)+a(t)},
 \qquad D_\omega(s,t)=(\partial_s+\partial_t)C_\omega(s,t).
\]
With the normalization \eqref{eq:coordinate},
\begin{align}
 8K_{\phi,\omega}(x,y)={}& C_\omega(2x,2y)B(2x)B(2y)\notag\\
 &+\int_0^\infty D_\omega(u+2x,u+2y)B(u+2x)B(u+2y)\dd u ,
 \qquad x,y\ge0.\label{eq:volterra}
\end{align}
The full theta series has $\nu_n=1$ for all $n$.
\end{proposition}

\begin{proof}
The mode identity \eqref{eq:theta} follows by two differentiations and the
change $v=2t$. Let $c_*$ denote the smallest $c_n$ with $\nu_n>0$ and
$\nu_*$ its weight. For every integer $k\ge0$,
\[
 \sum_n\nu_n c_n^k e^{-c_n e^v}
 \le e^{-c_*e^v}\sum_n\nu_n c_n^k e^{-(c_n-c_*)}\qquad(v\ge0),
\]
where the sum may be taken over the positive weights. The majorant is
finite, and $B(v)\ge\nu_*e^{v/4-c_*e^v}$. This gives locally uniform
convergence of the series and its differentiated series, and
$|B^{(k)}(v)/B(v)|=O(e^{kv})$. Also
$a(v)\ge c_* e^v-1/4$ by the positive weights. The denominator defining
$C_\omega$ is therefore bounded away from zero. The formulas for its first
derivatives are bounded on each ray $(s,t)=(u+2x,u+2y)$ by a polynomial
in $1+u$ times an exponential in $u$. Multiplication by $B(s)B(t)$ gives
an integrable majorant, and $C_\omega(s,t)B(s)B(t)\to0$ at infinity.

On that ray the derivative of $B(s)B(t)$ is
$-(a(s)+a(t))B(s)B(t)$. Therefore the derivative of
$C_\omega(s,t)B(s)B(t)$ is
\[
 D_\omega(s,t)B(s)B(t)-U\cosh(\omega U)\Psi(s)\Psi(t).
\]
Integrating on the ray and using \eqref{eq:coordinate} gives
\eqref{eq:volterra}, with the factor $8$ from $\frac12 z\dd z=U\dd u/8$.
The same majorant for the absolutely summed mode factors justifies expanding
their product under the original integral. Both mixed orders of a retained
core and its omitted modes are included when the full series is used.
\end{proof}

\section{Scope of the correction}\label{sec:scope}

On $H^8_0(0,8)$ the map $R$ in Section~\ref{sec:endpoint} is bounded into
$L^2(I)$ and the original smooth witness still belongs to its kernel.
An example of a well-posed revised minimization uses the altered form
\[
 \mathfrak a_\rho[f,g]=\mathfrak a_B[f,g]
       +\rho\langle f,g\rangle_{H^8},\qquad \rho>\|T_B\|_{L^2\to L^2},
\]
with the Sobolev norm including its $L^2$ term. The form defines an
equivalent Hilbert norm by boundedness and coercivity. On an attained trace
fiber, the orthogonal projection of zero in this Hilbert norm is the unique
minimizer $f_y$. Its Euler equation is
\begin{equation}\label{eq:euler}
 \mathfrak a_B[f_y,h]=-\rho\langle f_y,h\rangle_{H^8}
       \qquad(h\in\ker R).
\end{equation}
Thus the regularizer must be retained when the revised minimizer is used
in an identity for the original form. The Fourier restriction used below
is specified directly on evaluation vectors. Its identification with a
minimum-energy section would require a separate proof.

\section{Supplement: the shifted-\texorpdfstring{$\Xi$}{Xi} tests and their energies}\label{sec:evaluations}

Write $\xi(s)=\tfrac12s(s-1)\pi^{-s/2}\Gamma(s/2)\zeta(s)$ with its
removable values, and $\Xi(z)=\xi(1/2+iz)$. Fix $0<\omega<1/2$. The
evaluation core $\mathcal D$ consists of all finitely supported complex
coefficient functions $c$ on $\mathbb C_+=\{z:\operatorname{Im}z>0\}$.
For $c\in\mathcal D$ put
\[
 F_c(s)=\sum_z c_z e^{izs}\quad(s\in\R),\qquad Jc=F_c|_{(0,8)},\qquad
 E_\pm(z)=\Xi(z\pm i\omega).
\]
Define $H_\pm c(r)=(2\pi)^{-1/2}\sum_z c_z E_\pm(z)e^{izr}$ for $r\ge0$.
Each summand decays exponentially there. The required inequality is
\begin{equation}\label{eq:full-zeta}
 W_\omega(c)=\|H_+c\|_2^2-\|H_-c\|_2^2\ \ge 0
 \qquad(c\in\mathcal D).
\end{equation}
All norms of branches in this section are in $L^2(0,\infty)$.
Integrating each pair of exponentials gives
$W_\omega(c)=\sum_{z,w}c_z\overline{c_w}P_\omega(z,w)$, where
\begin{equation}\label{eq:pick}
 P_\omega(z,w)=\frac{E_+(z)\overline{E_+(w)}
                     -E_-(z)\overline{E_-(w)}}{2\pi i(\bar w-z)} .
\end{equation}
Thus the sign question is defined on the finite core before any completion.

\begin{corollary}[Negative endpoint-B energy on restricted evaluations]\label{cor:eval}
There exists $c\in\mathcal D$ supported on finitely many points of
$\{in:n=1,2,\ldots\}$ such that $\mathfrak a_B[Jc]<0$.
The same $c$ works for every shift in the stated range.
\end{corollary}
\begin{proof}
Under $t=e^{-s}$ these restrictions span $t\,p(t)$ with $p$ a polynomial
on $[e^{-8},1]$. Polynomials are dense in $H^8$ there: approximate the
eighth weak derivative in $L^2$ by polynomials and integrate eight times,
matching the traces of orders zero through seven at the left endpoint.
Repeated integration bounds the error in every lower derivative.
Multiplication by $t$ and the smooth change of variables both have bounded
inverses on the indicated Sobolev spaces. Hence some sequence of finite
restrictions converges in $H^8(0,8)$ to the witness $h$ of
Theorem~\ref{thm:negative}. Boundedness of the form on $L^2$ gives
$\mathfrak a_B[Jc_k]\to\mathfrak a_B[h]<0$. Take any sufficiently large
$k$. Both $J$ and $\mathfrak a_B$ are independent of $\omega$.
\end{proof}

\paragraph{Interior and exterior contributions.}
Use the full even theta profile $\Phi$ from \eqref{eq:theta}.
Its Fourier representation $\Xi(z)=\int_\R\Phi(t)e^{izt}\dd t$
follows by splitting the theta Mellin integral at one and integrating
twice by parts. For $\sigma=\pm1$ put
\[
 (A_\sigma f)(r)=\frac1{\sqrt{2\pi}}\int_0^8
      \Phi(s-r)e^{-\sigma\omega(s-r)}f(s)\dd s .
\]
Let $T_\sigma c$ be the same integral with $f=F_c$ over
$(-\infty,0)\cup(8,\infty)$. Substitution gives
$H_\sigma c=A_\sigma Jc+T_\sigma c$. All integrals converge since
$\Phi$ has every exponential moment. Each $A_\sigma$ is Hilbert--Schmidt
by integration of its squared kernel. With $f=Jc$, expansion gives
\begin{align*}
 W_\omega(c)-\mathfrak a_B[f]
 ={}&\big(\|A_+f\|^2-\|A_-f\|^2-\mathfrak a_B[f]\big)\\
 &+2\Rea\big(\langle A_+f,T_+c\rangle-\langle A_-f,T_-c\rangle\big)
   +\|T_+c\|^2-\|T_-c\|^2 .
\end{align*}
The first line is the bulk correction. The last two squared norms include
the cross terms between the two exterior intervals. This identity extends
by continuity in the graph norm with square
$\|Jc\|_{H^8}^2+\|T_+c\|^2+\|T_-c\|^2$.

The same branch integral identifies the role of the original Weyl kernel.
Writing $m=(s+t)/2$, $v=(s-t)/2$, the full-line coordinate kernel of
$W_\omega$ is
\[
 Q_\omega(s,t)=\frac1\pi\int_{|m|}^\infty
 \Phi(y+v)\Phi(y-v)\sinh(2\omega y)\dd y
 =\frac4\pi\int_0^\omega K_{\Phi,\eta}(s,t)\dd\eta .
\]
Indeed, subtract the two branch products and set $y=r-m$. The product
of theta factors is even in $y$, so the odd integrand has the same
integral from $-m$ as from $|m|$. The second equality follows by
integrating \eqref{eq:coordinate} in $\eta$. Pairing $Q_\omega$ with
$F_c(s)\overline{F_c(t)}$ on $\R^2$ gives $W_\omega(c)$.
Fubini applies termwise since each translated absolute branch integral
at $r$ is bounded by a constant times $e^{-r\operatorname{Im}z}$.

The distinction between the two topologies matters even for single
evaluations. For the unit coefficient at $in$, written $e_{in}$,
$\|Je_{in}\|_{H^8}^2=O(n^{15})$. By the functional equation,
\[
 4\pi n W_\omega(e_{in})=
 \xi(n+1/2+\omega)^2-\xi(n+1/2-\omega)^2 .
\]
Stirling's formula gives the ratio of the first $\xi$ value to the second
asymptotic to $(n/(2\pi))^\omega$, and each grows faster than every power
of $n$. Thus $W_\omega$ is unbounded on the unit ball of the restriction
norm. Corollary~\ref{cor:eval} determines the sign of the endpoint-B
contribution alone.

\begin{proposition}[A scalar test for the full energy]\label{prop:schur}
At fixed $\omega$, \eqref{eq:full-zeta} for all $c\in\mathcal D$ is
equivalent to
\begin{equation}\label{eq:diagonal}
 |\Xi(z-i\omega)|\le |\Xi(z+i\omega)|\qquad (z\in\mathbb C_+).
\end{equation}
It suffices to test \eqref{eq:diagonal} at points with rational real and
imaginary parts.
\end{proposition}
\begin{proof}
The diagonal of \eqref{eq:pick} is
$(|E_+(z)|^2-|E_-(z)|^2)/(4\pi\operatorname{Im}z)$, which proves
necessity. Conversely, \eqref{eq:diagonal} bounds $\Theta=E_-/E_+$ by
one off the isolated zeros of $E_+$. Each singularity in $\mathbb C_+$
is removable by boundedness. Multiplication by $\Theta$ is then a
contraction on $H^2(\mathbb C_+)$, as follows from the Hardy norm on
horizontal lines. For its reproducing kernel
$k(z,w)=1/[2\pi i(\bar w-z)]$, the adjoint of multiplication sends
$k(\cdot,w)$ to $\overline{\Theta(w)}k(\cdot,w)$. Positivity of
$I-M_\Theta M_\Theta^*$ therefore gives the positive kernel
$(1-\Theta(z)\overline{\Theta(w)})k(z,w)$. Multiplying by
$E_+(z)\overline{E_+(w)}$ gives \eqref{eq:pick} everywhere, including
common zeros, and proves sufficiency. Continuity gives the rational test.
\end{proof}

\paragraph{A finite check of the full energy.}
For $\omega=1/4$ and $z_n=in$, $1\le n\le16$, put
$b_n=\xi(n+1/2-\omega)/\xi(n+1/2+\omega)$.
The functional equation and \eqref{eq:pick} give the normalized real
matrix $[(1-b_m b_n)/(m+n)]$. The normalization divides out the
positive factors $E_+(im)E_+(in)/(2\pi)$.
The ancillary checker \texttt{verify\_shifted\_xi.py} encloses all
sixteen successive elimination pivots strictly above zero at 512 bits.
Thus $W_{1/4}$ is positive definite on this particular evaluation span.
Corollary~\ref{cor:eval} gives no upper bound on the nodes needed for
a negative endpoint-B restriction.

\paragraph{Sign for the family of shifts.}
The functional equation identifies $\Theta$ with the quotient in
Suzuki~\cite{suzuki}. It is inner for $\omega\ge1/2$ unconditionally,
and under RH for every $\omega>0$. Conversely, suppose $\Xi$ has a zero
$u$ in $\mathbb C_+$ (every nonreal zero gives one by conjugation).
For all sufficiently small $0<\omega<\min(1/2,\operatorname{Im}u)$,
isolation of $u$ gives $\Xi(u-2i\omega)\ne0$. At $z=u-i\omega$,
$E_+(z)=0$ and the singleton energy is strictly negative. Consequently
nonnegativity of the full energy for every required shift is equivalent
to RH. An unconditional sign proof in this range remains open here.

\end{document}